\documentclass[a4paper,12pt]{amsart}

\usepackage{amsmath,amsthm,amssymb}
\usepackage{hyperref}

\newcommand{\ti}{\widetilde}

\newcommand{\C}{\mathbb{C}}
\newcommand{\R}{\mathbb{R}}
\newcommand{\N}{\mathbb{N}}

\newcommand{\B}{\Big}

\allowdisplaybreaks[1]

\newtheorem{thm}{Theorem}
\newtheorem{lemma}{Lemma}
\newtheorem{cor}[thm]{Corollary}
\newtheorem{prop}[thm]{Proposition}
\newtheorem{example}[thm]{Example}
\newtheorem{defn}[thm]{Definition}
\theoremstyle{remark}
\newtheorem{rmk}[thm]{Remark}
\newtheorem{definition}[thm]{Definition}

\DeclareMathOperator{\Sing}{Sing}
\begin{document}

\title[nearly abelian transcendental semigroup]{dynamics of nearly abelian transcendental semigroup }
\author[R. Kaur]{Ramanpreet Kaur}
\address{Department of Mathematics, University of Delhi,
Delhi--110 007, India}

\email{preetmaan444@gmail.com}

\author[D. Kumar]{Dinesh Kumar}
\address{Department of Mathematics, Deen Dayal Upadhyaya College, University of Delhi,
Delhi--110 078, India}
\email{dinukumar680@gmail.com }
%
%

\begin{abstract}
The notion of nearly abelian rational semigroup was introduced by Hinkannen and Martin \cite{martin}. In this paper, we have introduced  the notion of nearly abelian transcendental semigroup. We have extended the results of nearly abelian rational semigroups to the best possible class of  nearly abelian transcendental semigroups. We  have given a class of functions which nearly permute with a given transcendental entire function. In addition, a relation between the postsingular set of composition of two entire functions with that of individual functions is obtained.
\end{abstract}



\maketitle
\section{introduction}\label{sec1}


A natural generalization of the dynamics associated to the iteration of a complex function is the dynamics of composite of two or more such functions and this leads to the realm of semigroups of transcendental entire functions.  Hinkkanen and Martin \cite{martin} did the seminal work in this direction. Their work was related to semigroups of rational functions. They extended  the dynamics  associated to the iteration of a rational function of one complex variable to the more general setting of an arbitrary semigroup of rational functions. Many of the results have been extended to semigroups of transcendental entire functions, see \cite{{cheng}, {dinesh1}, {dinesh3}, {dinesh4}, {poon1}, {zhigang}}. 
We recall some known facts of transcendental semigroup theory \cite{dinesh1}.
A transcendental semigroup $G$ is a semigroup generated by a family of transcendental entire functions $\langle{f_1,f_2,\ldots\rangle}$ with semigroup operation being functional composition. Moreover, $G$ generated by finitely many functions is said to be finitely generated.
The  Fatou set $F(G)$ of a transcendental semigroup $G$, is the largest open subset of $\C$ on which the family of functions in $G$ is normal and the Julia set $J(G)$ of $G$ is the complement of $F(G),$ that is, $J(G)=\ti\C\setminus F(G).$ 
\begin{definition}\label{sec1,defn1}
Consider a transcendental semigroup $G$. A set $U$ is said to be \emph{forward invariant} under $G$ if $g(U)\subset U$ for every $g\in G;$ and $U$ is said to be \emph{backward invariant} under $G$ if $g^{-1}(U)=\{w\in\C:g(w)\in U\}\subset U$ for each $g\in G.$ Also, $U$ is called \emph{completely invariant} under $G$ if it is both forward and backward invariant under $G.$ 
\end{definition}
In \cite[Theorem 2.1]{poon1}, it was shown that for a transcendental semigroup $G,$ $F(G)$ is forward invariant and $J(G)$ is backward invariant.


 In their paper, on the dynamics of rational semigroups I \cite{martin}, Hinkannen and Martin introduced the notion of  nearly abelian rational semigroups. They  established several properties of nearly abelian rational semigroups for instance, the existence of no wandering domains for Fatou set of a nearly abelian rational semigroup. Motivated from their work, we now introduce nearly abelian transcendental semigroups.  This paper is  divided into three sections. First section is introduction where we have given outline of the paper. In the second section, we have extended the results of nearly abelian rational semigroups to the best possible class of  nearly abelian transcendental semigroups. Moreover, we have extended some of the results of \cite{dinesh1, dinesh3, dinesh4} to nearly abelian transcendental semigroups. In the final section, we have given a class of functions which nearly permute (definition of nearly permutable entire functions to be given in Section \ref{sec3}) with a given transcendental entire function using Wiman-Valiron theory. In addition, we provide a class of nearly abelian transcendental semigroup and obtain a relation between the postsingular set of composition of two entire functions with the postsingular set of individual functions.

\section{nearly abelian transcendental semigroup}\label{sec2}
We now introduce the notion of nearly abelian transcendental semigroup.
\begin{defn}
Let $G=\langle f_1,f_2,\ldots\rangle $ be a transcendental semigroup .
  $G$ is said to be a nearly abelian semigroup, if there exist a  precompact family  of affine maps (say) $ \Phi=\{\phi_\alpha\}, \alpha\in\Lambda$(index set) that satisfies the  following properties:
\begin{enumerate}
\item $ \phi(F(G)=F(G)$ for every $\phi\in \Phi;$
\item for every $f,g \in G$ there exists $ \phi\in \Phi $ such that $f\circ g=\phi\circ g\circ f.$
\end{enumerate}
\end{defn}

 Associated with a nearly abelian semigroup is the \emph{commutator set}:
\begin{defn}
 Let $G$ be a nearly abelian semigroup . The commutator set of $G$ is defined as:\\
 $\Phi(G)=\{\phi: \text{there are}\  f,g\in G  \text{ such that}\  f\circ g=\phi\circ g\circ f\}$
\end{defn}

 Given a  transcendental entire function $f,$  for $n\in\N$ let $f^n$ denote the $n$-th iterate of $f.$
In the following theorem, we establish a relation between Fatou set of a nearly abelian semigroup with Fatou set of each element in the semigroup.
\begin{thm}
	Let \(G=\langle f_1,f_2,\dots ,f_n\rangle\) be a finitely generated nearly abelian transcendental semigroup such that \(\infty\) is not a limit function of any subsequence in G in a component of \(F(G)\). Then \(F(G)=F(g)\) for every \(g\in G\).
	\end{thm}
	\begin{proof}
Let \(f\in G\) be fixed and \(g\in G\) be arbitrary. 
We shall show that \(F(f)=F(g)\) which will prove the theorem. 
We first establish that \(g(F(f))\subset F(f)\).  To this end, 
let \(z_0\in F(f)\) and \(U\) be a neighborhood of \(z_0\) such that \(\overline{U}\subset F(f)\). Also,  \(g(U)\) is  a neighborhood of \(g(z_0)\). 
As \(G\) is nearly abelian semigroup therefore, for each \(n\geq 1\) there exists \(\phi_n\) such that \(f^n\circ g=\phi_n\circ g\circ f^n\). 
Also, \(z_0\in F(f)\) therefore, there exists a subsequence say \(\{m_j\}\) such that \(f^{m_j}\to \eta\)(say) as \(j\to\infty\)
uniformly on \(U\). 
This implies that \(g\circ f^{m_j}\to g\circ \eta\) on \(U\). 
Also, \(\Phi(G)\) is a compact family of affine maps therefore, there exists a subsequence say \(\{n_j\}\) 
such that \(\phi_{n_j}\to \phi\) as \(j\to\infty\).
On combining we get, \(f^{m_j}\circ g\to \phi\circ g\circ\eta\) on \(U\). 
Therefore, the family \(\{f^{n}\circ g\}\) is normal on \(U\). 
Hence, \(g(F(f))\subset F(f)\). 
As a consequence of Montel's theorem, we have \(F(f)\subset F(g)\). 
On similar lines, we obtain that \(F(g)\subset F(f)\). 
Hence \(F(G)=F(g)\) for every \(g\in G\). 
	\end{proof}
Recall that Sing($f^{-1}$) denotes the set of critical values and asymptotic values of transcendental entire function $f$ and their finite limit points.
We now show that under some conditions on the singularities of the generators of a finitely generated nearly abelian semigroup, Julia set of the semigroup equals Julia set of its generators.
	\begin{thm}
Let \(G=\langle f,g\rangle\) be  a finitely generated nearly abelian transcendental semigroup such that \(\Sing(f^{-1})\) 
and \(\Sing(g^{-1})\) are bounded. Then \(J(G)=J(f)=J(g)\). 
\end{thm}
\begin{proof}
Firstly, we observe that \(g^{-1}(I(f))\subset I(f)\). 
For this, let \(a\in g^{-1}(I(f)\)) which implies that \(g(a)\in I(f)\),
i.e. \(f^{n}(g(a))\to\infty\) as \(n\to\infty\). As \(G\) is nearly abelian semigroup,
 for each \(n\in \mathbb{N}\), there exists an affine map \(\phi_n\in\Phi(G)\) 
such that \( f^{n}(g(a))=\phi_n(g(f^n(a)\to \phi(g(f^n(a)\) as \(n\to\infty\). 
Hence, \(f^{n}(a)\to\infty\) as \(n\to\infty\) which implies that \(a\in I(f)\).
 Now, assume that  \(a\in I(f)\) be such that \(a\) is not a Fatou exceptional value of \(f\).
 From above observation, we have \(g^{-n}(I(f))\subset I(f)\) for every \(n \geq 1\) 
 which implies that \(\cup_{n=1}^{\infty}g^{-n}(f)\subset I(f)\). 
 Using the hypothesis, we have \(J(f)=\overline {I(f)}\) \cite{el2} and \(J(g)\subset \overline {\cup_{n=1}^{\infty}g^{-n}(a)}\),
  where \(a\) is not a Fatou exceptional value of \(g\).
  Therefore, \(J(g)\subset \overline {\cup_{n=1}^{\infty}g^{-n}(a)}\subset \overline {I(f)}= J(f)\). On similar lines, one obtains \(J(f)\subset J(g)\) which proves the result.
\end{proof}
Recall that a transcendental entire function $f$ is of Speiser class if \(\Sing(f^{-1})\) is finite. We say that a transcendental semigroup $G$ is of Speiser class if each $g\in G$ is  of Speiser class.
The next result provides a class of nearly abelian transcendental semigroup for which the behavior of Fatou components of the semigroup and those of individual functions are alike.
\begin{thm}\label{thm3}
	Let \(G\) be a nearly abelian transcendental semigroup of Speiser class. 
	Then behavior of components of \(F(G)\) and \(F(f)\) for every \(f\in G\) is alike.
\end{thm}
\begin{proof}
It is enough to show that for a stable basin $U$ of $F(G),$  \(G_{U}=G\) (see \cite{dinesh1} for more details on classification of perodic components of $F(G)$ of a transcendental semigroup $G$). 
We have, \(G_U\subset G\). 
For the backward implication, suppose  there exists a \(g\in G\) such that \(g\not\in G_U\), i.e. \(g(U)\not\subset U\). 
As \(G\) is a nearly abelian transcendental semigroup, so we have \(F(g)=F(h)\) for all \(g,h\in G\) and \(h(U)\subset U\). 
Also, \(g,h\in S \), so behavior of components of \(F(g)\) and \(F(h)\) are alike \cite{ren}, which is a contradiction to our supposition and hence proves the result.
\end{proof}

 It is well known that functions in Speiser class do not have wandering domains, \cite{el2, keen}.
As an immediate consequence of Theorem \ref{thm3}, we have absence of wandering domains for a nearly abelian transcendental semigroup of Speiser class.
\begin{cor}
	Let \(G\) be a nearly abelian transcendental semigroup of Speiser class. 
	Then \(G\) does not have wandering domains.
\end{cor}
By a result of Baker \cite{baker1'}, a transcendental entire function has at most one completely invariant component. Another consequence of Theorem \ref{thm3} is extension of this result of Baker to nearly abelian transcendental semigroup of Speiser class.
\begin{cor}
Let \(G\) be a nearly abelian transcendental semigroup of Speiser class. Then \(F(G)\) has at most one completely invariant component.
\end{cor}
\begin{proof}
Suppose that \(F(G)\) has two completely invariant components, say \(U_1\) and \( U_2\).
 Because \(F(G)=F(f)\) for every \(f\in G\) and also behavior of components are alike, we have 
  \(U_1\) and \(U_2\) are completely invariant components of $F(f)$ for every \(f\in G\),  which is a contradiction and hence proves the result.
\end{proof}

The second author \cite{dinesh1} established that for a finitely generated abelian transcendental semigroup, if all stable domains of Fatou set of the generators are bounded then Fatou set of the semigroup does not contain any asymptotic values. We have  similar result for a nearly abelian semigroup.
\begin{thm}
Let \(G= \langle f_1,f_2,\ldots,f_n\rangle\) be a finitely generated nearly abelian transcendental semigroup with \(\Phi(G)=\{\phi:\phi(z)=az+b, |a|=1\}\). 
If all stable domains of \(F(f_i), 1 \leq i\leq n\) are bounded, then \(F(G)\) does not contain any asymptotic values of $G$.
\end{thm}
\begin{proof}
It is enough to show that \(F(g\circ f)\subset F(f)\cap F(g)\). 
It is already known that \(z\in F(f\circ g)\) if and only if \(g(z)\in F(g\circ f)\) \cite{berg4}. 
As \(G\) is a nearly abelian transcendental semigroup, therefore, for \(f,g\in G\) there exists \(\phi\in\Phi(G)\)
 such that \(f\circ g=\phi\circ g\circ f\) which implies that \(z\in F(\phi\circ g\circ f)\) if and only if \(g(z)\in F(g\circ f)\). 
As \(\phi(z)=az+b\), we have \(\phi^n(z)=a^nz+c\) for some \(c\in \C\) implying that \(F(\phi\circ g\circ f)=F(g\circ f)\) as \(|a|=1\). 
Hence, we have \(z\in F(g\circ f)\) if and only if \(g(z)\in F(g\circ f)\), i.e. \(g(F(g\circ f))\subset F(g\circ f)\). 
On similar lines, we get \(f(F(g\circ f))\subset F(g\circ f)\). 
Therefore, we have \(F(g\circ f)\subset F(f)\) and \(F(g\circ f)\subset F(g)\), i.e. \(F(g\circ f)\subset F(f)\cap F(g)\). 
Now, the above proved fact implies that all stable domains of \(F(g),g\in G\) are bounded. 
Suppose \(z_0\in F(G)\) is an asymptotic value, then \(z_0\) is an asymptotic value of  \(F(g)\), for some  $g\in G.$  But this is a contradiction to the fact that if \(g\) is a transcendental entire function whose all stable domains are bounded, then \(F(g)\) does not contain any asymptotic values \cite{Hua}.
\end{proof} 
The next result gives a class of  finitely generated nearly abelian semigroup for which we have equality of Fatou set of the generators.
\begin{thm}
	Let \(G=\langle f, g\rangle\) where \(g=af+b\), $(0\not=)\, a\in \C, b\in \R^+$ be  a nearly abelian transcendental semigroup with \(\Phi(G)= \{\phi: \phi(z)=z+c,c\in\mathbb{R^+}\}\). Then \(F(f)=F(g)\). Also, $|a|=1$.
\end{thm}
\begin{proof}
	It is sufficient to show that \(g(F(f))\subset F(f)\). 
	Suppose  \(z_0\in F(f)\). We need to prove that  \(g(z_0)\in F(f)\). As \(z_0\in F(f)\),  so there exists a suitable neighborhood \(U(say)\) of \(z_0\) such that \(\overline U\subset F(f)\). 
	If \{\(f^n\)\} converges to a holomorphic function say \(\eta\) on \(U\), then \{\(f^n\)\} is normal on \(g(U)\) 
	and hence \(g(U)\subset F(f)\) and we are done in this case. 
	Now, suppose that \(f^n\to \infty\) as \(n\to \infty \) on \(U\). 
	Then for any \(A>0 \left(A>\frac{1+|b|}{|a|}\right)\) there exists \(n_0\in\mathbb{N}\) such that \(|f^n(z)|>A\), for all \(n\geq n_0\) and $z\in U.$ 
	For \(n>n_0\), we have
\begin{equation}\notag
\begin{split}
|f^n(g(z))|
&=|\phi_n(g(f^n(z)))|\\
&=|g(f^n(z))+c_n|\\
&\geq |af^{n+1}(z)+b+c_n|\\
&\geq |af^{n+1}(z)+b|\\
&\geq|a|A-|b|\\
&>1.
\end{split}
\end{equation}
	Then, by Montel's normality criterion, \(\{f^n\}\) is normal on \(g(U)\). 
	Therefore, \(g(F(f))\subset F(f)\). 
	Combining both the cases, we obtain \(F(f)=F(g)\). Now, we prove that \(|a|=1\). We have \(f\circ g(w)=\phi\circ g\circ f(w)\) with \(g=af+b\). 
  Letting \(f(w)=z\), we obtain \(f(az+b)=af(z)+b+c\). 
  Suppose \(a\neq 1\), and 
	put $z=y+\frac{b}{1-a}$ and $h(y)=f(y+\frac{b}{1-a})-\frac{b}{1-a}-\frac{c}{1-a}$. Then 
\begin{equation}\notag
\begin{split}
f(az+b)
&=f\big(ay+\frac{b}{1-a}\big)\\
&=h(ay)+\frac{b}{1-a}+\frac{c}{1-a},
\end{split}
\end{equation}
 and 
\begin{equation}\notag
\begin{split}
af(z)+b+c
&=af\big(y+\frac{b}{1-a}\big)+b+c\\
&=ah(y)+\frac{b}{1-a}+\frac{c}{1-a}.
\end{split}
\end{equation}
Therefore, we have \(h(ay)=ah(y)\). 
Considering Taylor series of \(h\), namely, \(h(z)=  {\sum}_{k=0}^{k=\infty}r_kz^k\), we obtain 
	 \(h(ay)= {\sum}_{k=0}^{k=\infty}r_k a^ky^k\) and \(ah(y)={\sum}_{k=0}^{k=\infty}ar_ky^k\). 
	On comparing coefficients of like terms, we get \(a(a^{k-1}-1)=0\) which implies that \(|a|=1\).
\end{proof}

Bergweiler and Wang \cite{berg4} proved that for two entire functions $f$ and $g$, \(\Sing((f\circ g)^{-1})\subset \Sing (f^{-1})\cup f(\Sing (g^{-1}))\).
We now establish a similar relation in the nearly abelian situation.
\begin{prop}
 Let $f$ and $g$ be transcendental entire functions such that \(f\circ g=\phi\circ g\circ f\). 
Then \(\Sing((f\circ g)^{-1})=\phi( \Sing((g\circ f)^{-1}))\). 
\end{prop}
\begin{proof}
We have 
\begin{equation}\notag
\begin{split}
\Sing((f\circ g)^{-1})
&=\Sing((\phi\circ g\circ f)^{-1})\\
&\subseteq \Sing(\phi)^{-1}\cup \phi(\Sing((g\circ f)^{-1})\\
&=\phi(\Sing((g\circ f)^{-1})
\end{split}
\end{equation}
which implies the forward implication.
For the backward implication, observe that
\begin{equation}\notag
\begin{split}
 \phi(\Sing((g\circ f)^{-1}))
&=\phi (\Sing((\phi^{-1}\circ f\circ g)^{-1}))\\
&\subseteq \phi((\Sing(\phi^{-1})^{-1})\cup \Sing ((f\circ g)^{-1})\\
&=\Sing ((f\circ g)^{-1})).
\end{split}
\end{equation} 
Combining the two implications, we obtain the desired result.
\end{proof}

The notion of conjugate transcendental semigroup was defined in \cite{dinesh1}. We extend this notion to conjugate nearly abelian transcendental semigroup.
\begin{defn}
 Let \(G=\langle g_1,g_2,\ldots,g_n\rangle \) be a finitely generated transcendental semigroup. 
 Consider \(G'=\langle \eta \circ g_1\circ \eta^{-1},\eta \circ g_2\circ \eta^{-1},\ldots,\eta \circ g_n\circ \eta^{-1}\rangle\). Then  $G'$ is called conjugate nearly abelian  transcendental semigroup where \(\eta: \C\to \C\) is  the conjugate map, $\eta(z)=az+b, 0\neq a, b\in \C.$ 
\end{defn}
It can be easily seen that \(\eta (F(G))=F(G')\). To see this, we have \(F(G)=F(g_i)\), $1\leq i\leq n.$ 
Also \(F(G')= F(\eta \circ g_i \circ \eta^{-1}), 1\leq i\leq n\). 
Therefore, 
\begin{equation}\notag
\begin{split}
\eta (F(G))
&= \eta (F(g_i))\\
&=F(\eta \circ g_i \circ \eta^{-1})\\
&=F(G').
\end{split}
\end{equation} 
In \cite{martin}, it was shown that for a nearly abelian rational semigroup, there exists a Mobius transformation satisfying some relation under composition with elements of the semigroup. The following result gives existence of an affine map for a nearly abelian transcendental semigroup satisfying similar kind of relation under composition  with elements of the semigroup.  
\begin{thm}
	Let \(G\) be  a nearly abelian transcendental semigroup. 
	Then for all \(f\in G\) and for all \(\phi\in \Phi(G)\), there exists an affine map \(\eta\) such that \(f(\phi(z))=\eta( f(z))\) for almost all values of \(z\). 
 	\end{thm}
 \begin{proof}
 	As \(\phi\in \Phi(G)\), therefore there exists  \(g_1,g_2\in G\) such that \(g_1\circ g_2=\phi\circ g_2\circ g_1\). 
 	Then, we have \(f\circ g_1\circ g_2=f\circ \phi\circ g_2\circ g_1\). 
 	Also, there are \(\phi_1, \phi_2\in\Phi(G)\) such that \(f\circ g_1\circ g_2=\phi_1\circ g_1 \circ f \circ g_2=\phi_1\circ\phi_2 \circ f \circ g_2\circ g_1 \). Since 
 	\(f, g_1\) and $g_2$ are transcendental entire functions, so
 	\(f, g_1\) and $g_2$ can leave at most one value.
 	Hence, we have \(f\circ \phi=\eta\circ f\), where \(\eta=\phi_1\circ\phi_2\).
 \end{proof}
\begin{rmk}
Hinkannen and Martin \cite{martin} proved that for a nearly abelian rational semigroup $G,$ and \(\phi\in \Phi(G)\),  it is not always possible to find an \(\eta\) such that \(\phi\circ f= f\circ \eta\). We now show by an example, the same result might also not hold when  $G$ is a nearly abelian transcendental semigroup.
\end{rmk}
\begin{example}
 Let \(f(z)=e^{z^2},g(z)=-e^{z^2}\) and \(G=\langle f,g\rangle\). 
Taking \(\phi(z)=-z\) we obtain  \(\phi\circ f=f\) which implies that
 \(G\) is a nearly abelian semigroup. 
Now consider
\begin{equation}\notag
\begin{split}
 f\circ g
&= f\circ\phi\circ f\\
&=f^2\\
&=\phi^2\circ f^2\\
&=\phi\circ g\circ f.
\end{split}
\end{equation}
This implies that \(\Phi(G)=\{\phi, Id\}\), where Id denotes the identity function.
Hence \(\phi\circ f=-f\neq f=f\circ \eta\) for every \(\eta\in \Phi(G)\).
\end{example}
 Beardon \cite{beardon'} established that if \(f\) and \(g\) are polynomials satisfying \(J(f)=J(g)\), then there is a linear mapping \(\phi(z)=az+b\) satisfying  \(f\circ g= \phi\circ g\circ f\) and \(|a|=1\). The following example shows this result  need not to be true in the case of transcendental entire functions.
\begin{example}
 Let \(f(z)=e^{\lambda z}\) and  $g(z)= e^{\lambda z}+\frac{1}{\lambda}$. 
One can choose \(\lambda \) such that \(F(f)=F(g)\). 
It can be easily verified that for \(\phi(z)= ez-\frac{e}{\lambda}\), \(f\circ g= \phi\circ g\circ f\).
\end{example}
 
We now give an example of nearly abelian transcendental semigroup.
\begin{example}

 Let \(f(z)=e^{\lambda z}\) and \(g(z)=f(z)+\frac{2\pi\iota}{\lambda}\), where \(\lambda\in\left(0,\frac{1}{e}\right)\). 
Consider the semigroup \(G=\langle f,g\rangle\).
Then by definition of \(f\) and \(g\), any \(h\in G\) is of the form \(f^{m}\) for some $m\in\N$ or \(f^{k}+\frac{2\pi\iota}{\lambda}\), for some $k\in\N.$
We now show that \(G\) is nearly abelian semigroup.
To this end, consider \(\Phi(G)=\{\phi, Id\}\), where \(\phi(z)=z-\frac{2\pi\iota}{\lambda}\). 
It can be easily seen that \(F(f)=F(g)\) and hence, \(\eta(F(G))=F(G)\) for every \(\eta\in \Phi(G)\). 
Next, let \(h_1,h_2\in G\). Then as discussed above,  \(h_1= f^m\) and \(h_2=f^k+\frac{2\pi\iota}{\lambda}\) for some \(m,k\in \mathbf{N}\). 
Thus we obtain \(h_1\circ h_2(z)=f^{m+k}\) and \(h_2\circ h_1(z)=f^{m+k}+\frac{2\pi\iota}{\lambda}\). 
Hence, \(h_1\circ h_2(z)= \phi\circ h_2\circ h_1(z)\). 
This establishes that  $G$ is nearly abelian semigroup.
\end{example}
\section{special class of entire functions}\label{sec3}
In this section, we provide a class of entire functions which nearly permute with a given transcendental entire function using Wiman-Valiron theory. In addition, we give a class of nearly abelian transcendental semigroup and obtain a relation between the postsingular set of composition of two entire functions with the postsingular set  of individual functions. The results in this section are motivated from the paper \cite{zheng}. In fact, the proofs given here goes verbatim as  in \cite{zheng} with few minor changes. For sake of completeness of the paper, we reproduce here the proofs. We now give the definition of nearly commuting (or nearly permutable) entire functions.
\begin{defn}
Two transcendental entire functions \(f\) and $g$ are said to be nearly permutable (or nearly commuting) if there exist an affine function $\phi: \C\to\C$ satisfying 
\[ g\circ f= \phi\circ f\circ g.\]
\end{defn}
To prove our results, we have used the notion of value distribution theory of meromorphic functions, also known as Nevanlinna Theory \cite{yang}. For an entire function $f$, the \emph{order} and \emph{lower order} of $f$ are defined respectively,  in the following manner:

$$ \rho(f) =  \limsup_{r \rightarrow \infty} \frac{\log^+ \log^+ M(r, f)}{\log r}=\limsup_{r \rightarrow \infty} \frac{\log^+ T(r, f)}{\log r},$$

$$\mu(f) =\liminf_{r\rightarrow \infty} \frac{\log^+ \log^+ M(r, f)}{\log r}=\liminf_{r \rightarrow \infty} \frac{\log^+ T(r, f)}{\log r},$$

where $ M(r,f)= \max\{\ |f(z)|:|z| =r \}\ $ is the maximum modulus of $f(z)$ over the circle $|z| =r$ and $T(r,f)$ is the characteristic function of $f(z).$ 
We now state a lemma which will be heavily used in the proofs of  results which are to follow.
\begin{lemma}\label{sec3,lem1}\cite{zheng}
	Suppose \(F_0(z),F_1(z),\ldots, F_m(z)\) are $m+1$ entire functions not vanishing identically and \(h_0(z),h_1(z),\ldots,h_m(z)\,\,(m\geq1)\) are arbitrary meromorphic functions not all identically zero. 
Also, let \(g(z)\) be  a non constant entire function, \(K\)  a positive real number and \(\{r_j\}\) an unbounded monotone increasing sequence of positive real numbers such that, for each \(j\),
	\[ T(r_j, h_i)\leq T(r_j, g)\,\,\, (i=0,1,2,\ldots,m)\]
	\[T(r_j, g^{(1)})\leq (1+o(1))T(r_j, g).\]
	 If \(F_i(z)\) and \(h_i(z)(i=0,1,\ldots,m)\) satisfy
	 \[F_0(g)h_0+F_1(g)h_1+\ldots+F_m(g)h_m=0,\] 
	 then there exists polynomials \(P_0(z),P_1(z),\ldots,P_m(z)\) not all identically zero such that \[F_0(z)P_0(z)+F_1(z)P_1(z)+\ldots+F_m(z)P_m(z)=0.\]
\end{lemma}

Under some conditions on the  regularity of two entire functions which nearly commute, we obtain the following result:
\begin{thm}\label{sec3,thm1'}
	Let \(f\) and $g$ be transcendental entire functions with finite order and \(f\) with positive lower order satisfying \(g\circ f= \phi\circ f\circ g\), where \(\phi(z)=az+b\), \(|a|\geq 1, b\in\mathbb{R^+}\cup \{0\}\). 
	Then, if \(f\) satisfies a differential equation with polynomial coefficients then so does \(g\).
\end{thm}
\begin{proof}

As \(g\circ f=af(g)+b\), using Polya's theorem \cite{polya}, we obtain \[M(M(r,g),f)\geq M(r,g(f))=M(r,af(g)+b)>M(r,f(g))> M(cM(\frac{r}{2},g),f)\] where \(c\in(0,1)\). 
Since \(\rho(g) < \infty\) and \(\mu(f)>0\), there exists positive numbers \(K_1\) and \(K_2\) such that \[K_1 \log M(r,f)>\log\log M(cM\left(\frac{r}{2},g\right),f)>K_2 \log M\left(\frac{r}{2},g\right)\geq K_2 T\left(\frac{r}{2},g\right).\] 
By a result from \cite{hayman}, for each positive number \(r\)
 and any real number \(R>r\), we have \(\log^{+}\log^{+} M(r,f)< \frac{R+r}{R-r}\,T(r,f)\). 
 This implies that \(\log^{+} M(r,f)< 3 T(2r,f)\). 
 Hence, \(T(r,g)<K_3 T(4r,f)\) for some positive number \(K_3\). 
 Since \(\rho(f)<\infty\), there exists Polya's peak \(\{r_j\}\) of \(f\). 
 Namely, there exists three sequences \(\{r_j^{1}\},\{r_j^{2}\},\{\epsilon_j\}\) satisfying \(\{r_j^{1}\to + \infty\},\{ \frac{r_j}{r_j^{1}}\to + \infty\}, \{ \frac{r_j^{2}}{r_j}\to + \infty\}, \{\epsilon_{j}\to 0\}\) as \(j\to \infty\) and when \(r_j^{1}\leq t \leq r_j^{2}\), \[ T(t,f)< (1+\epsilon_j)\left(\frac{t}{r_j}\right)^{\rho(f)} T(r,f).\] 
 Also, from the condition \(\rho(f)<+\infty\), we have \(m(r,\frac{f^{(k)}}{f})= O(\log r) \,\, (k=1,2,\dots)\). 
 Hence \(T(r,f)\leq (1+o(1)) T(r,f)\). 
 On combining the above relations, we obtain \(T(r_j,g^{(k)})< 2K_3 4^{\rho(f)} T(r_j,f)\), for \((k=1,2,\dots)\). 
 Suppose now that \(f\) satisfies a differential equation with polynomial coefficients:
 \[P_0(z)f^{(n)}(z)+P_1(z)f^{(n-1)}(z)+\cdots+P_n(z)f(z)+P_{n+1}(z)=0.\]
 As \(f(g(z))= a g(f(z))+b\), 
on differentiating this equation and after substituting the values of  \(f^n(g),f^{n-1}(g),\dots,f(g)\), we get an equation of the form, 
\(h_0g^{(n)}(f)+\dots+h_ng(f)+h_{n+1}=0\), where all of \(h_0,h_1,\dots,h_{n+1}\) are  rational functions 
in $z,f,f^{(1)},\dots,f^{(n)},\dots,\allowbreak g,g^{(1)},\dots,g^{(n)}$. 
From above discussion, we see that all conditions of  Lemma \ref{sec3,lem1} are satisfied. Hence, we obtain the required result. 
\end{proof}

In the following result, we give a class of entire functions which nearly permute with a given transcendental entire function.
\begin{thm}
Let \(f(z)= \sin z + q(z)\), where \(q(z)\) is a non constant polynomial. Suppose \(g\) is a transcendental entire function satisfying \(f\circ g=\phi \circ g\circ f\), where \(\phi(z)=az+c, |a|>1, c\in \mathbb{R}^+\). Then \(g(z)= r f(z)+s\) for some constants \(0\neq r,s\in\mathbb{R}\).	
\end{thm}
\begin{proof}
We have \(f(z)= \sin z + q(z)\), so that 
\begin{align}
f''(z)+f(z)&=q(z)+q''(z),\\
(f(z)-q(z))^{2}+(f'(z)-q'(z))^{2}&=1.
\end{align} 

On differentiating \(f(g)=\phi(g(f))\), we obtain
\begin{align}
f'(g)g'&= ag'(f)f',\\
f''(g)g'^{2}+f'(g)g''&=a\left(g''(f)f'^{2}+g'(f)f''\right).
\end{align}
Using these equations we obtain
\begin{align}
a\left(\frac{f'}{g'}\right)^{2}g''(f)+ \left(a \frac{f''}{g'^{2}}-\frac{f'g''}{g'^{3}}\right) g'(f)+ag(f)&= q(g)+q''(g)-c,\\
\left(ag(f)+c-q(g)\right)^{2}+\left(\left(\frac{f'}{g'}\right)g'(f)-q'(g)\right)^{2}&=1.
\end{align} 
Using  \((1)\) and Theorem \ref{sec3,thm1'}, there exists polynomials \(P_0,P_1,P_2\) not all identically zero such that 
\begin{align}
P_0(z)g''(z)+P_1(z)g'(z)=P_2(z).
\end{align} 
Similarly, using \((2)\) and Theorem \ref{sec3,thm1'}, there exists polynomials \(Q_0,Q_1,Q_2,Q_3,Q_4\) not all identically zero such that \begin{align}
Q_0(z)g^{2}(z)-2Q_1(z)g(z)+Q_2(z)g'^{2}-2Q_3(z)g'(z)=Q_4(z).
\end{align} 
We claim that \(Q_0(z)\not\equiv 0\). 
On the contrary, suppose that \(Q_0 \equiv 0\). Therefore, \((8)\) becomes
\[-2 Q_1g+Q_2 g'^{2}-2Q_3g'=Q_4.\]
 From \((7),(8)\), we obtain
 \begin{align}
 P_0(f)g''(f)+P_1(f)g'(f)+P_2(f)g(f)\\
=P_3(f)Q_0(f)g^{2}(f)-2Q_1(f)g(f)+Q_2(f)'g^{2})(f)-2Q_3(f)g'(f)=Q_4(f)\notag
\end{align}

On eliminating \(g''(f)\) from \((9)\) and \((5)\)
we get,
\begin{align}
\left[P_0(f)\left(\frac{af''}{g'^{2}}-\frac{f'g''}{g'^{3}}\right)-aP_1(f)\left(\frac{f'}{g'}\right)^2\right]g'(f)&+\left[aP_0(f)-aP_2(f)\left(\frac{f'}{g'}\right)^2\right]g(f)\\
\nonumber 
&= P_0(f)(A(g)-c)-aP_3(f)\left(\frac{f'}{g'}\right)^2. 
\end{align}
Now eliminating, \(g^{2}(f)\) from \((10)\) and \((6)\) we obtain,
\begin{align}
\left[2aQ_3(f)-2Q_0(f)Q'(f)\left(\frac{f'}{g'}\right)\right]g'(f)&+\left[2aQ_1(f)-2a(c-Q(g))Q_0(f)\right]g(f)+\\
 \nonumber \left[ \left(\frac{f'}{g'}\right)^2Q_0(f)-2aQ_2(f)\right]g'^{2}(f)&=Q_0(f)(1-Q'^{2}(g)-(c-Q(g))^2)-2aQ_4(f).
\end{align}
If \(P_0(f)\neq P_2(f)\left(\frac{f'}{g'}\right)^2\), then eliminating \(g(f)\) from  \((10)\) and \((11)\) and applying Lemma \ref{sec3,lem1} on \(g'(f)\),
 we get \(\left(\frac{f'}{g'}\right)^2Q_0(f)-2aQ_2(f)\), i.e.
\(g'^2=R(f)f'^{2}\), where \(R(f)\) is a  rational function. As an application of  Wiman-Valiron theory , we get that \(R(f)\) is  a constant function. Therefore, \(g= rf +s\). 
If \(P_0(f)= P_2(f)\left(\frac{f'}{g'}\right)^2\), then again on similar lines, we get the desired result.
\end{proof}

Recall that the postsingular set of an entire function $f$ is defined as 
\[P(f)=\overline{\B(\bigcup_{n\geq 0}f^n(\text{Sing}(f^{-1}))\B)}.\]
The final result of this section gives a class of nearly abelian transcendental semigroup. It also gives a relation between the postsingular set of composite of two entire functions with the postsingular set of individual functions.
\begin{thm}
 Let $G=\langle f,g\rangle$ be a transcendental semigroup where $f$ and $g$ are periodic with same period say $c.$ Suppose $f$ and $g$ are nearly commuting, i.e.  $f\circ g= \phi \circ g\circ f$ where $\phi(z)=z+c$. Then $G$ is nearly abelian transcendental semigroup. Moreover, 
 \[P(f\circ g)\subset \overline{\cup_{n\geq 0}\,\phi\circ g^n P(f)\cup \cup_{n\geq0}f^{n+1}P(g)}\] 
 \end{thm}
\begin{proof}
Firstly, we shall prove that $G$ is nearly abelian transcendental semigroup. For this, it is sufficient to show that\\
$f^n\circ g=\phi \circ g\circ f^n, f\circ g^n=\phi \circ g^n\circ f, f^n\circ g^m=\phi \circ g^m \circ f^n$ for every $n,m\in \N$. We will prove our result by induction. \\
Case 1: We need to show that $f^n\circ g=\phi \circ g\circ f^n$ for every $n\in \N $.\\
For $n=1$, it is given that $f\circ g=\phi \circ g\circ f$. Assume that $f^n\circ g=\phi \circ g\circ f^n$ and we shall prove that $f^{n+1}\circ g=\phi \circ g\circ f^{n+1}$. Now, 
\begin{equation}\notag
\begin{split}
f^{n+1}\circ g
&=f\circ f^n\circ g\\
&= f \circ \phi \circ g \circ f^n\\
&=f\circ g \circ f^n\\
&=\phi \circ g \circ f^{n+1}. 
\end{split}
\end{equation}
Hence, Case 1 is proved.\\
Case 2: We need to show that $f\circ g^n=\phi \circ g^n\circ f$ for every $n\in \N $.\\
For $n=1$, it is given that $f\circ g=\phi \circ g\circ f$. Assume that $f\circ g^n=\phi \circ g^n\circ f$. Now, 
\begin{equation}\notag
\begin{split}
 f \circ g^{n+1}
&= f \circ g^n \circ g\\
&= \phi \circ g^n\circ f \circ g\\
&= \phi \circ g^n\circ \phi \circ g \circ f\\
&= \phi \circ g^{n+1}\circ f. 
\end{split}
\end{equation}
 Hence, result of this case also follows.\\
On combining  Case 1 and Case 2, we obtain $f^n \circ g^n=\phi \circ g^n \circ f^n$ for every $n \in \N$.\\
Case 3: In this case, we need to show that $f^n\circ g^m=\phi \circ g^m \circ f^n$ for each $n,m\in \N$.  Let $n>m$, i.e. $n=m+k $ for some $k\in \N$. Therefore, 
\begin{equation}\notag
\begin{split}
f^n\circ g^m
&=f^k\circ f^m\circ g^m\\
&=f^k \circ \phi\circ g^m \circ f^m\\
&= f^k \circ g^m \circ f^m\\
&=\phi \circ g^m \circ f^n.
\end{split} 
\end{equation}
Therefore, $G$ is nearly abelian transcendental semigroup with $\Phi(G)=\{Id, \phi,\phi^{-1}\}.$\\
We now show that $P(f\circ g)\subset \overline{\cup_{n\geq 0}\,\phi\circ g^n P(f)\cup \cup_{n\geq0}f^{n+1}P(g)}$. It can be easily seen that $(f\circ g)^n \subseteq \phi \circ g^n P(f)\cup f^{n+1} P(g)$ for every $n \in \N$.
We know that
$\Sing(f \circ g)^{-1})\subset \Sing(f^{-1})\cup f\Sing(g^{-1})$. Therefore, $f\circ g(\Sing(f \circ g)^{-1}))\subset  f \circ g(\Sing(f^{-1}))\cup f\circ g\circ f(\Sing(g^{-1}))$. Using the given facts, we get $f\circ g(\Sing(f \circ g)^{-1})\subset \phi \circ g (P(f)) \cup f^2 (P(g))$. By induction we have $ (f\circ g)^n\subset \phi \circ g^n(P(f))\cup f^{n+1}(P(g))$ for every $n\in \N$. 
Hence, we obtain the desired relation
\[P(f\circ g)\subset \overline{\cup_{n\geq 0}\,\phi\circ g^n P(f)\cup \cup_{n\geq0}f^{n+1}P(g)}.\qedhere\]

\end{proof}

\end{document}